\documentclass[11pt,a4paper]{article}
\usepackage{amsmath,amsthm}
\usepackage{amsfonts,amstext,amssymb, amsbsy, mathtools, comment, graphicx, url}

\newcommand{\bs}{\boldsymbol}
\newcommand{\p}{\mathbb P}
\newcommand{\e}{\mathbb E}
\newcommand{\D}{\mathrm d}
\newcommand{\M}{\mathcal{M}}
\newcommand{\levy}{L\'{e}vy }
\newcommand{\C}{\mathbb C}
\newcommand{\R}{\mathbb R}
\renewcommand{\Re}{{\rm Re}}
\newcommand{\Cpos}{\mathbb C^{\Re>0}}
\newcommand{\Cneg}{\mathbb C^{\Re<0}}

\newcommand{\diag}{\mathrm {diag}}
\newcommand{\ind}[1]{\mbox{\large  1}_{\{#1\}}}

\renewcommand{\a}{\alpha}
\newcommand{\lmb}{\lambda}
\newcommand{\Lmb}{\Lambda}
\newcommand{\G}{\Gamma}
\newcommand{\ba}{{\bs a}}
\newcommand{\bsg}{{\bs \sigma}}
\newcommand{\bv}{{\bs v}}
\newcommand{\bx}{{\bs x}}
\newcommand{\be}{{\bs e}}
\newcommand{\bu}{{\bs u}}
\newcommand{\bl}{{\bs \ell}}
\newcommand{\bh}{{\bs h}}
\newcommand{\bk}{{\bs k}}
\newcommand{\1}{{\bs 1}}
\newcommand{\0}{{\bs 0}}
\newcommand{\bpi}{{\bs \pi}}
\newcommand{\bup}{{\mbox{$\bu_+$}}}

\newcommand{\bud}{{\mbox{$\bu_\downarrow$}}}
\newcommand{\blp}{{\mbox{$\bl_+$}}}
\newcommand{\blm}{{\mbox{$\bl_-$}}}
\newcommand{\blu}{{\mbox{$\bl_\uparrow$}}}

\newcommand{\Ep}{{\mbox{$E_+$}}}
\newcommand{\Em}{{\mbox{$E_-$}}}
\newcommand{\Eu}{{\mbox{$E_\uparrow$}}}
\newcommand{\Ed}{{\mbox{$E_\downarrow$}}}

\newcommand{\Np}{{\mbox{$N_+$}}}
\newcommand{\Nm}{{\mbox{$N_-$}}}
\newcommand{\Nu}{{\mbox{$N_\uparrow$}}}
\newcommand{\Nd}{{\mbox{$N_\downarrow$}}}
\newcommand{\matI}{\mathbb{I}}
\newcommand{\matO}{\mathbb{O}}

\newcommand{\Lmbp}{\mbox{$\Lmb^+$}}
\newcommand{\Lmbm}{\mbox{$\Lmb^-$}}
\newcommand{\Gp}{\G^+}
\newcommand{\Gm}{\G^-}

\newcommand{\Vpp}{V^+_+}

\newcommand{\Vmm}{V^-_-}
\renewcommand{\t}{'}

\newtheorem{theorem}{Theorem}
\newtheorem{lemma}[theorem]{Lemma}
\newtheorem{proposition}[theorem]{Proposition}
\newtheorem{corollary}[theorem]{Corollary}
\theoremstyle{definition}
\newtheorem{definition}{Definition}

\begin{document}
\bibliographystyle{plain}
\title{First passage process of a {M}arkov additive process, \\with applications to reflection problems}
\author{Bernardo D'Auria\footnotemark[1]
, Jevgenijs Ivanovs\footnotemark[2]
, Offer Kella\footnotemark[3]\hspace{4pt} and Michel Mandjes\footnotemark[2]}
\footnotetext[1]{Universidad Carlos III de Madrid, Avda Universidad 30, 28911 Leganes (Madrid), Spain.}
\footnotetext[2]{E{\sc urandom}, P.O. Box 513, 5600 MB Eindhoven, the Netherlands; Korteweg-de Vries Institute for Mathematics, University of Amsterdam, Science Park 904, 1098 XH Amsterdam, the Netherlands.}
\footnotetext[3]{Department of Statistics, The Hebrew University of Jerusalem, Jerusalem 91905, Israel.}
\date{\today\ (first version: June 19, 2009)}
 \maketitle
\begin{abstract}
In this paper we consider the first passage process of a
spectrally negative Markov additive process (MAP).
The law of this process is uniquely characterized by
a certain matrix function, which plays a crucial role in fluctuation theory.
We show how to identify this matrix using the theory of Jordan chains
associated with analytic matrix functions. Importantly, our result also provides
us with a technique,
which can be used to derive various further identities. We
then proceed to show how to compute the stationary distribution
associated with a one-sided reflected (at zero) MAP for both the
spectrally positive and spectrally negative cases as well as for
the two sided reflected Markov-modulated Brownian motion; these
results can be interpreted in terms of queues with MAP input.
\end{abstract}

\vspace{0.1in}\noindent \hspace{-0.1in}\begin{tabular}{l l}
AMS 2000 Subject classification:&Primary 60K25\\&Secondary 60K37\end{tabular}\\
Keywords: \levy processes; fluctuation theory; Markov additive
processes; Markov-modulated Brownian motion

\normalsize

\section{Introduction}
Continuous-time Markov additive processes (MAPs) with one-sided
jumps have proven to be an important modelling tool in various
application areas, such as communications networking \cite[Ch.\
6-7]{prabhu} and finance
\cite{asmussen:phaselevy2004,jobertrogers:optionmod2006}. Over the
past decades a vast body of literature has been developed; see for
instance \cite[Ch.\ XI]{asmussen:apq} for a collection of results.
A MAP can be thought of as
a \levy process whose Laplace exponent depends on the state of a
(finite-state) Markovian background process (with additional jumps
at transition epochs of this background process). It is a
non-trivial generalization of the standard \levy process, with
many analogous properties and characteristics, as well as new
mathematical objects associated to it, posing new challenges.
Any \levy process is characterized by a Laplace exponent $\psi(\a)$; its
counterpart for MAPs is the matrix exponent
$F(\a)$, which is essentially a multi-dimensional analogue of $\psi(\a)$.

In this paper we consider the {\em first passage process} $\tau_x$
defined as the first time the process exceeds level $x$.
We concentrate on the case of a  \emph{spectrally negative} MAP (that is,
all jumps are negative), so that  the first passage process
is a MAP itself. Knowledge of the matrix exponent of this process,
which we in the sequel denote by the matrix function $\Lambda(q)$,
is of crucial interest when addressing related fluctuation
theory issues. Indeed it can be considered as the multi-dimensional generalization of $-\Phi(q)$, where $\Phi(q)$ is the (one-dimensional) right-inverse  of $\psi(\a)$, as given in \cite[Eqn.\ (3.15)]{kyprianou}.
Our main result concerns the identification of the matrix function $\Lambda(q)$ in
terms of the matrix exponent $F(\a)$ of the original MAP. We provide the
Jordan normal form of $\Lambda(q)$ relying  on the theory of Jordan chains
associated with analytic matrix functions.

Importantly, our main result is not only about identification of the
matrix exponent of the first passage process. We prefer to see our contribution rather as the development of a new
{\em technique}: the theory of analytic matrix functions, combined with
the special structure of the Jordan pairs of
$F(\a)-q\matI$ (with $\matI$ being the identity
matrix), and their relation to the matrix
$\Lambda(q)$, enables the derivation of a set of further identities. These identities, such as~(\ref{eq:idnty1}) and~(\ref{eq:idnty2}), then play an important role
in the solution of a number of problems related to the {\em reflection} of the MAP; here `reflection' amounts to an adaptation
of the MAP in order to ensure that the process attains values in a certain subset of ${\mathbb R}$ only.
In this sense, we could reflect  at 0 to obtain a process that assumes
nonnegative values only (which can be interpreted as a {\em queue} with infinite buffer capacity), or at both 0 and $b>0$ to restrict the process to $[0,b]$
(this {\em double reflection}  is essentially a queue with finite buffer $b$).  In
the following we discuss in more detail these reflection problems, which we
solved using this technique.

In \cite{asm_kella_multidimensional}  a martingale associated with
the MAP was established, being essentially the multidimensional
counterpart of the martingale found in \cite{kellawhitt} for the standard
\levy process. In the same paper, various situations were
considered in which this martingale could be applied. Most
notably, attention was paid to reflection of a {\em spectrally positive}
MAP at 0, i.e.,
a queue fed by a {spectrally
positive} MAP; here `spectrally positive' means that all jumps are positive. For this model the
multidimensional martingale allowed to reduce the problem of
determining the stationary workload to finding various constants
that are solutions to a  system of linear equations. The
authors of \cite{asm_kella_multidimensional} did not succeed,
however, in proving that such a system of equations has a unique
solution. Also for the situation of doubly reflected
Markov-modulated Brownian motion, in the following denoted by MMBM,
%(that is, a finite-buffer queue fed by Markov-modulated Brownian motion),
a similar complication arose.
In the literature, problems of this type were only partially addressed for
special cases
(e.g., see \cite{asmussen:mmbm, extremes_dieker, karandikar:kulkarni:1995, kella_stadje}).
In our paper we tackle these problems using the technique outlined above.

This paper is organized as follows. Section 2 reviews some main
results from analytic matrix function theory, while in Section 3
we identify the matrix exponent $\Lmb(q)$ by relating the Jordan
pairs of the matrix functions $F(\a)-q\matI$ and $\a\matI+\Lmb(q)$
for a fixed $q\geq 0$. This result, which is
Theorem~\ref{thm_main} and which can be considered as the main
contribution of our work, is explicit in the sense that it is
given in terms of computable quantities associated with $F(\a)$.
As mentioned above, it in addition leads to a technique which can be used to obtain
further important identities. This technique is discussed in
Section 4 and is then applied in Sections 5-8, where we solve a
number of open problems related to reflected processes. In Section
5 we study spectrally one-sided MAPs reflected at 0. In
particular, both for spectrally negative and spectrally positive
MAP input, we express the steady-state workload in terms of
quantities related to the matrix exponent $\Lambda(q)$ of the
first passage process of a spectrally negative MAP. The result of
Section 5.1 is not new and is here for completeness. In Section
5.2 we succeed in solving the (above mentioned) issues that
remained open in \cite{asm_kella_multidimensional}. In Section 6
we apply the methodology to identify the stationary distribution
of Markov-modulated Brownian motion with two reflecting barriers.
We provide a full treatment of this model, also for the case of 0
asymptotic drift. In Section 7 we identify the so-called scale functions associated with MMBM,
which is, to our best knowledge, a new result.
Yet another demonstration of applicability of our technique is
given in Section 8, where we present a simple proof of the fact
that $\Lmb(q)$ is a unique solution of a certain matrix integral
equation. This result (in different degrees of generality) appears
in \cite{asmussen:mmbm, extremes_dieker, miyazawa, pacheco_prabhu,
pistorius,rogers}, and is commonly considered as the main tool
used to numerically identify $\Lmb(q)$. Some spectral
considerations (under the assumption that $\Lmb(q)$ has distinct
eigenvalues) can be found in~\cite{asmussen:mmbm,
extremes_dieker,pacheco_prabhu}. Finally, an iterative method to
compute $\Lmb(q)$ can be found in~\cite{breuer}.

In order to point out the contribution of our work as precisely as possible,
we feel that the following comment is important.
It should be realized that it is the connection between MAPs and the theory of
Jordan chains that enables the most general treatment of the problem. As we
mentioned above, we do not need to assume that certain systems of equations have
a unique solution, and hence we feel that from a mathematical perspective our
framework is more natural than those used before. One could argue, however, that
in practical situations the uniqueness requirement will virtually always be
fulfilled. This concern is valid if one is only concerned with the
computation of $\Lambda(q)$, but the importance of our result becomes apparent
when considering applications such as those presented in Section 5-8. For
example, the special but important case of the MMBM with zero drift immediately
leads to a non-simple
eigenvalue $0$. In this case an additional equation associated to the null Jordan
chain is required to obtain the solution, see~(\ref{eq:2sided_0drift})
and~(\ref{eq:scale_lost}). Finally, the proofs of identities such
as~(\ref{eq:idnty1}), (\ref{eq:idnty2}), (\ref{eq:C_D}) and (\ref{eq:mat_eq})
can be obtained in a routine way using the technique presented in Section 4,
which we can do only due to the full generality of the main result; this was not
possible before.

The remainder of this section is devoted to the definition of the quantities of
interest, with a focus on spectrally negative MAPs and their first passage process.

\subsection{Spectrally negative MAP}
A MAP is a bivariate Markov process
$(X(t),J(t))$ defined as follows. Let
$J(\cdot)$ be an irreducible continuous-time Markov chain with
finite state space $E=\{1,\ldots,N\}$, transition rate matrix
$Q=(q_{ij})$ and a (unique) stationary distribution $\bpi$. For each
state $i$ of $J(\cdot)$ let $X_i(\cdot)$ be a \levy process with
Laplace exponent $\psi_i(\a)=\log(\e e^{\a X_i(1)})$. Letting
$T_n$ and $T_{n+1}$ be two successive transition epochs of
$J(\cdot)$, and given that $J(\cdot)$ jumps from state $i$ to
state $j$ at $T_n$, we define the additive process $X(\cdot)$ in
the time interval $[T_n,T_{n+1})$ through
\begin{equation}
X(t)=X(T_n-)+U_{ij}^n+[X_j(t)-X_j(T_n)],
\end{equation}
where $(U_{ij}^n)$ is a sequence of independent and identically
distributed random variables with moment generating function
\begin{equation}
\tilde{G}_{ij}(\alpha) = \e e^{\a
U_{ij}^1},\quad \hbox{ where } \quad U_{ii}^1 \equiv 0,
\end{equation}
describing the jumps at transition epochs.
To make the MAP spectrally negative, it is required that
$U_{ij}^1\leq 0$ (for all $i,j\in\{1,\ldots,N\}$) and that
$X_i(\cdot)$ is allowed to have only negative jumps (for all
$i\in\{1,\ldots,N\}$). As a consequence, the moment generation functions $\tilde{G}_{ij}(\alpha)$ are well defined
for $\alpha\ge 0.$

A \levy process is called a \emph{downward subordinator} if it has
non-increasing paths a.s. We denote the subset of indices of $E$ corresponding to such processes by $\Ed$.
Let also $\Ep=E\backslash\Ed$, $\Nd=|\Ed|$ and $\Np=|\Ep|$. It is convenient to assume that $\Ep=\{1,\ldots,\Np\}$, which we do throughout this work.
We use $\bv_+$ and $\bv_\downarrow$ to
denote the restrictions of a vector $\bv$ to the indices from
$\Ep$ and $\Ed$ respectively. Finally, in order to
exclude trivialities it is assumed that $\Np>0$.

Define the matrix $F(\a)$ through
\begin{equation}\label{generator}
F(\a)=Q\circ\tilde G(\a)+\diag[\psi_1(\a),\ldots,\psi_N(\a)],
\end{equation}
where $\tilde G(\a)=(\tilde G_{ij}(\a))$; for matrices $A$ and
$B$ of the same dimensions we define  $A\circ B=(a_{ij}b_{ij})$. One can see
that in the absence of positive jumps $F(\a)$ is analytic on
$\Cpos=\{\a\in\C:\Re(\a)>0\}$. Moreover, it is known that
\begin{equation}\label{eq:init}
\e_i[e^{\a X(t)};J(t)=j]:=\e_i[e^{\a X(t)}\ind{J(t)=j}]=(e^{F(\a)t})_{ij},
\end{equation}
cf.\ \cite[Prop.\ XI.2.2]{asmussen:apq}, where $\e_i(\cdot)$
denotes expectation given that $J(0)=i.$
We also write $\e[e^{\a X(t)};J(t)]$ to denote the matrix with $ij$-th element given in~(\ref{eq:init}).
Hence $F(\a)$ can be seen
as the multi-dimensional analog of a Laplace exponent, defining
the law of the MAP. In the following we call $F(\a)$the  \emph{matrix exponent} of the MAP $(X(t),J(t))$.

An important quantity
associated to a MAP is the \emph{asymptotic drift}:
\begin{equation}
\kappa=\lim_{t\rightarrow\infty}\frac{1}{t}\e_i
X(t)=\sum_i\pi_i\left(\psi_i'(0)+\sum_{j\neq i}q_{ij}\tilde
G_{ij}'(0)\right),
\end{equation}
which does not depend on the initial state $i$ of
$J(t)$~\cite[Cor.\ XI.2.7]{asmussen:apq}. Finally for $q\geq 0$ we
define $F^q(\a)=F(\a)-q\matI$, with $\matI$ being the identity
matrix, which can be seen as the matrix exponent of the MAP `killed' at random time $e_q$:
\begin{equation}
\e[e^{\a X(t)};t<e_q,J(t)]=e^{(F(\a)-q\matI)t},
\end{equation}
where $e_q$ is an exponential random variable of rate $q$
independent of everything else and $e_0\equiv \infty$ by
convention.

\subsection{First Passage Process}
Define the first passage time over level $x>0$ for the (possibly
killed) process $X(t)$ as
\begin{equation}\tau_x=\inf\{t\geq 0: X(t)> x\}.\end{equation}
It is known that on $\{J(\tau_x)=i\}$ the process $(X(t+\tau_x)-X(\tau_x),J(t+\tau_x)),t\geq 0$
is independent from $(X(t),J(t)),t\in[0,\tau_x]$ and has the same law as the original process under $\p_i$.
Therefore, in the absence of positive jumps the
time-changed process $J(\tau_x)$ is a time-homogeneous Markov
process and hence is a Markov chain. Letting $\{\partial\}$ be an
absorbing state corresponding to $J(\infty)$, we note that
$J(\tau_x)$ lives on $\Ep\cup\{\partial\}$, because $X(t)$ can not
hit new maximum when $J(t)$ is in a state corresponding to a
downward subordinator; see also \cite{palmowski_fluctuations}.
Let $\Lmb(q)$ be the $\Np\times \Np$ dimensional transition rate
matrix of $J(\tau_x)$ restricted to $\Ep$, that is
\begin{equation}\label{lambda_matrix}\p(J(\tau_x)=j,\tau_x<e_q\mid J(\tau_0)=i)=(e^{\Lmb(q) x})_{ij},\hbox{
where }i,j\in \Ep.\end{equation}
It is easy to see that in the absence of positive jumps the first passage process $(\tau_x,J(\tau_x)), x\geq 0$
is a MAP itself. Moreover,  \[\e[e^{-q \tau_x};J(\tau_x)=j\mid J(\tau_0)=i]
=\p(J(\tau_x)=j,\tau_x<e_q\mid J(\tau_0)=i)=(e^{\Lmb(q) x})_{ij},\]
so that $\Lmb(q)$ is the matrix exponent of (the negative of) the first passage
process. This interpretation, however, is not used in the rest of this paper.

Another matrix of interest is
$N\times \Np$ matrix $\Pi(q)$ defined by
\begin{equation}\label{pi_matrix}
\Pi(q)_{ij}=\p_i(J(\tau_0)=j,\tau_0<e_q),\hbox{ where }i\in E\hbox{
and }j\in \Ep.
\end{equation}
This matrix specifies initial distributions of the time-changed
Markov chain $J(\tau_x)$. Note also that $\Pi(q)$ restricted to the
rows in $\Ep$ is the identity matrix, because $\tau_0=0$ a.s.\
when $J(0)\in\Ep$ \cite[Thm.\ 6.5]{kyprianou}.
We note that the case of $q=0$ is a special case corresponding to no killing.
In order to simplify notation we often write $\Lmb$ and $\Pi$ instead of
$\Lmb(0)$ and $\Pi(0)$.

It is noted that if $q>0$ or $q=0,\kappa<0$ then $\Lmb(q)$ is a
defective transition rate matrix:
  $\Lmb(q) \1_+\leq \0_+$, with at least one strict inequality.
If, however, $\kappa\geq 0$, then $\Lmb$ is a non-defective
transition rate matrix:  $\Lmb \1_+=\0_+$; also $\Pi \1_+=\1$.
These statements follow trivially from~\cite[Prop.\
XI.2.10]{asmussen:apq}. Finally, note that $\Lmb$ is an
irreducible matrix, because so is $Q$. Hence if $\Lmb$ is
non-defective then by Perron-Frobenius theory \cite[Thm.\
I.6.5]{asmussen:apq} the eigenvalue $0$ is simple, because it is
the eigenvalue with maximal real part.

It is instructive to consider the `degenerate' MAP, i.e., the one with dimension $N=1$. Such a MAP is just a \levy process,
and $\Lambda(q)=-\Phi(q)$, where $\Phi(q)$ is the right-inverse of $\psi(\a),\a\geq 0$.
Note also that $\Lmb$ being non-defective (and hence singular) corresponds to $\Phi(0)=0$.

\section{Preliminaries}
In this section we review some basic facts from analytic matrix function theory.
Let $A(z)$ be an analytic matrix function ($n
\times n$ dimensional), defined on some
domain $D\subset \C$, where it is assumed that $\det(A(z))$ is not identically zero on
this domain. For any $\lmb \in D$  we can write
\begin{equation}A(z)=\sum_{i=0}^\infty \frac{1}{i!}A^{(i)}(\lmb) (z-\lmb)^i,\end{equation}
where $A^{(i)}(\lmb)$ denotes the $i$-th derivative of $A(z)$
at~$\lmb$. We say that $\lmb$ is an \emph{eigenvalue} of $A(z)$ if
$\det(A(\lmb))=0$.
\begin{definition}
We say that vectors $\bv_0,\ldots,\bv_{r-1}\in\C^n$ with
$\bv_0\neq \0$ form a \emph{Jordan chain} of $A(z)$ corresponding
to the eigenvalue $\lmb$ if
\begin{equation}\label{dfn:jordan}
\sum_{i=0}^j
\frac{1}{i!}A^{(i)}(\lmb)\bv_{j-i}=\0,\hspace{0.5in}j=0,\ldots,r-1.
\end{equation}
\end{definition}
Note that this definition is a generalization of the well-known
notion of a Jordan chain for a square matrix $A$. In this classical
case $A(z)=z \matI - A$, and~(\ref{dfn:jordan}) reduces to
\begin{equation}\label{eq:jord.ch}
A\bv_0=\lmb \bv_0,\hspace{0.2in}A\bv_1=\lmb \bv_1+\bv_0,\hspace{0.2in}\ldots,\hspace{0.2in}A\bv_{r-1}=\lmb \bv_{r-1}+\bv_{r-2}.
\end{equation}

The following result is well known~\cite{gohberg_rodman} and is an immediate consequence of (\ref{eq:jord.ch}).
\begin{proposition}\label{jordan_lemma}Let $\bv_0,\ldots,\bv_{r-1}$ be a Jordan chain of $A(z)$ corresponding to the eigenvalue $\lambda$, and let $C(z)$ be $m\times n$ dimensional matrix.
If $B(z)=C(z)A(z)$ is $r-1$
times differentiable at $\lmb$, then
\begin{equation}\label{eq:jordan_lemma}\sum_{i=0}^j\frac{1}{i!} B^{(i)}(\lmb)\bv_{j-i}=\0,\hspace{0.5in}j=0,\ldots,r-1.\end{equation}
\end{proposition}
Note that if $B(z)$ is a square matrix then
$\bv_0,\ldots,\bv_{r-1}$ is a Jordan chain of $B(z)$ corresponding
to the eigenvalue~$\lmb$. It is, however, not required that
$C(z)$ and $B(z)$ be square matrices.

Let $m$ be the multiplicity of $\lmb$ as a zero of $\det(A(z))$
and $p$ be the dimension of the null space of $A(\lmb)=A_0$. It
is known, see e.g.~\cite{gohberg_rodman}, that there exists a
\emph{canonical system of Jordan chains} corresponding to
$\lmb$
\begin{equation}
\bv_0^{(k)},\bv_1^{(k)},\ldots,\bv_{r_k-1}^{(k)},\hspace{0.5in} k=1,\ldots,p,
\end{equation}
such that the vectors $\bv_0^{(1)},\ldots,\bv_0^{(p)}$ form the
basis of the null space of $A_0$ and $\sum_{i=1}^p r_i=m$. We
write such a canonical system of Jordan chains in matrix form:
\begin{equation}\label{eg:jordan_pair}
V=[
\bv_0^{(1)},\bv_{1}^{(1)},\ldots,\bv_{r_1-1}^{(1)},
\ldots,
\bv_0^{(p)},\bv_{1}^{(p)},\ldots,\bv_{r_p-1}^{(p)}
],\hspace{0.2in}
\Gamma=\diag[\Gamma^{(1)},\ldots,\Gamma^{(p)}],\end{equation}
where $\Gamma^{(i)}$ is the Jordan block of size $r_i\times r_i$ with
eigenvalue~$\lmb$, i.e. a square matrix having zeros everywhere except along the diagonal, whose elements are equal to $\lmb$, and the superdiagonal, whose elements are equal to $1$.
\begin{definition}
A pair of matrices $(V,\Gamma)$ given by~$(\ref{eg:jordan_pair})$ is
called a \emph{Jordan pair} of $A(z)$ corresponding to the
eigenvalue~$\lmb$.
\end{definition}
We note that, unlike in the classical case, the vectors
forming a Jordan chain are {\em not} necessarily linearly independent;
furthermore a Jordan chain may contain a null vector.

We conclude this section with a result on entire functions of matrices defined
through
\begin{equation}\label{eq:entire}
f(M)=\sum_{i=0}^\infty \frac{1}{i!}f^{(i)}(0)M^i,
\end{equation}
for an entire function $f:\C\rightarrow\C$ and a square matrix $M$. The next
lemma will be important for applications.
\begin{lemma}\label{jordan_lemma_2}
Let $f:\C\rightarrow\C$ be an entire function and let $\Gamma$ be a Jordan block of size $k$ with $\lmb$ on the diagonal, then for an arbitrary set of vectors $\bv_0,\ldots,\bv_{k-1}$ the
$(j+1)$-st column of the matrix $[\bv_0,\ldots,\bv_{k-1}]f(\Gamma)$ equals
\begin{equation}\sum_{i=0}^j\frac{1}{i!}f^{(i)}(\lmb)\bv_{j-i},\end{equation}
where $j=0,\ldots,k-1$.
\end{lemma}
\begin{proof}
Immediate from~\cite[Thm.\ 6.6]{analytic_functions_of_mat}.
\end{proof}

\section{Jordan Normal Form of $\Lmb(q)$}
In this section we consider a spectrally negative MAP
$(X(t),J(t))$ with matrix exponent $F(\a)$ and asymptotic drift
$\kappa$. Let $\lmb_1,\ldots,\lmb_k$ be the eigenvalues of
$F^q(\a)$, to be understood as the zeros of $\det(F^q(\a))$, for a given $q\ge 0$, in its region of
analyticity $\Cpos$. Let $(V_i,\G_i)$ be a Jordan pair
corresponding to the eigenvalue $\lmb_i$. Define the matrices $V$
and $\G$ in the following way:
\begin{equation}\label{eq:Vdef}
  \begin{array}{rll}
    V=&[V_1,\ldots,V_k]\\
    \G=&\diag[\G_1,\ldots,\G_k] &\quad \hbox{if }q>0\hbox{ or }q=0,\kappa< 0; \vspace{.25cm} \\
    V=&[\1,V_1,\ldots,V_k]\\
    \G=&\diag[0,\G_1,\ldots,\G_k]& \quad \hbox{if }q=0,\kappa\geq 0.
  \end{array}
\end{equation}
and let the matrices $V_+$ and $V_\downarrow$ be the restrictions
of the matrix $V$ to the rows corresponding to $\Ep$ and
$\Ed$ respectively.

\begin{theorem}\label{thm_main}
It holds that $\G$ and $V_+$ are $\Np\times
  \Np$-dimensional matrices, $V_+$ is invertible, and
\begin{equation}\label{lmb_dec}
\Lmb(q)=-V_+\G V_+^{-1},\hspace{0.5in}\Pi(q)=VV_+^{-1}.
\end{equation}
\end{theorem}

We start by establishing a lemma, which can be considered as a
weak analog of Thm.~\ref{thm_main}.
\begin{lemma}\label{eig_lemma}
If $\bv^0,\ldots,\bv^{r-1}$ is a Jordan chain of $F^q(\a)$
corresponding to the eigenvalue $\lmb\in\Cpos$ then
$\bv^0_+,\ldots,\bv^{r-1}_+$ is a Jordan chain of $\a \matI
+\Lmb(q)$ corresponding to the eigenvalue~$\alpha=\lmb$ and
$\Pi(q) \bv^i_+=\bv^i$ for $i=0,\ldots,r-1$.
\end{lemma}
\begin{proof}
It is known from \cite[Theorem 2.1]{asm_kella_multidimensional} or by direct
applying the Dynkin's formula, see~\cite{ethier:kurtz:1986},
that for $\a\in\Cpos$
\[
M_\a(t)=\left[\int_0^{t} e^{\a X(s)}\be\t_{J(s)}\D s\right]\cdot
F(\a)+\be\t_k-e^{\a X(t)}\be\t_{J(t)},
\]
is a row vector valued zero mean martingale under the probability
measure $\p_k$, where `$\,\t\,$' denotes the transposition
operation.
Apply the optional sampling theorem to % the martingale
$M_\a(\cdot)$ %, $\a\in\Cpos$ as defined in~Lemma~\ref{lem:martingale} in the Appendix,
with the finite stopping time $t\wedge\tau_x\wedge e_q$ and note
that \[\e_k \left[e^{\a X(e_q
)}\ind{t\wedge\tau_x>e_q}\be\t_{J(e_q)}\right]=q\e_k\left[\int_0^{t\wedge\tau_x\wedge
e_q }e^{\a X(s)}\be\t_{J(s)}\D s\right],\] to obtain
\begin{equation}\label{EQ}
C(\alpha) F^q(\a)=B(\a),\end{equation} where
\begin{equation}\nonumber
C(\a)=\e_k\left[\int_0^{t\wedge \tau_x\wedge e_q} e^{\a
X(s)}\be\t_{J(s)}\D s\right],\:\:B(\a)=\e_k\left[e^{\a X(t\wedge
\tau_x)}\ind{t\wedge\tau_x<e_q}\be\t_{J(t\wedge
\tau_x)}\right]-\be\t_k.
\end{equation}
Equation (\ref{EQ}) can be seen as the Wiener-Hopf factorization for
the killed MAP $(X(t),J(t))$, see also \cite{kaspi:1982}.

Noting that $X(\cdot)\leq x$ on $[0,\tau_x]$ and using usual
dominated convergence argument we conclude that $B(\a)$ is
infinitely differentiable in $\a\in\Cpos$. Apply
Prop.~\ref{jordan_lemma} to (\ref{EQ}) to see that for all
$j=0,\ldots,r-1$ the following holds true:
\begin{equation}\nonumber
\sum_{i=0}^j\frac{1}{i!}\e_k\left[X^i(t\wedge \tau_x)e^{\lmb X(t\wedge \tau_x)}\ind{t\wedge\tau_x<e_q}\be\t_{J(t\wedge \tau_x)}\right]\bv^{j-i}-\be\t_k\bv^j=0.
\end{equation}
Letting $t\rightarrow\infty$ we obtain
\begin{equation}\label{eq:pi_lambda}\sum_{i=0}^j\frac{1}{i!}x^ie^{\lmb
x}\p_k(J(\tau_x),\tau_x<e_q)\bv^{j-i}-\be\t_k\bv^j=0,\end{equation}
where $\p_k(J(\tau_x),\tau_x<e_q)$ denotes a row vector with
$\ell$-th element given by $\p_k(J(\tau_x)=\ell,\tau_x<e_q)$. Note that
the case when $q=0$ and $\p_k(\tau_x=\infty)>0$ should be treated
with care. In this case $\kappa<0$ and thus
$\lim_{t\rightarrow\infty}X(t)=-\infty$ a.s.\ \cite[Prop.\
XI.2.10]{asmussen:apq}, so the above limit is still valid.

Considering~(\ref{eq:pi_lambda}) for all $k\in E$ and
choosing $x=0$ we indeed obtain
$\Pi(q) \bv^j_+=\bv^j.$
If, however, we pick $k\in \Ep$, then
\begin{equation}\sum_{i=0}^j\frac{1}{i!}x^ie^{(\lmb \matI+\Lmb(q))x}\bv^{j-i}_+-\bv^j_+=\0_+.\end{equation}
Take the right derivative in $x$ at $0$ of both sides to see that
\begin{equation}(\lmb \matI+\Lmb(q))\bv^j_++\bv^{j-1}_+=\0_+,\end{equation}
which shows that
$\bv^0_+,\ldots,\bv^{r-1}_+$ is a Jordan chain of $\a \matI +\Lmb(q)$
corresponding to the eigenvalue~$\lmb$.
\end{proof}

We are now ready to give a proof of our main result, Thm.~\ref{thm_main}.
\begin{proof}\textit{of Theorem~\ref{thm_main}\quad}
Lemma~\ref{eig_lemma} states that $\bv^0_+,\ldots,\bv^{r-1}_+$ is a
classical Jordan chain of the matrix $-\Lmb(q)$. Recall that if
$q=0,\kappa\geq 0$ then $\Lmb(q) \1_+=\0_+$ and $\Pi(q)
\1_+=\1$. Therefore the columns of $V_+$ are linearly
independent \cite[Prop.\ 1.3.4]{gohberg:lancaster:rodman:2006} and
\begin{equation}-\Lmb(q) V_+=V_+\G,\hspace{0.5in}\Pi(q) V_+=V.\end{equation}

Consider the case when $q>0$. Now \cite[Thm.\ 1]{zeros} states that
$\det(F^q(\a))$ has $N_+$ zeros (counting multiplicities)
in~$\Cpos$; see also \cite[Rem.\ 1.2]{zeros}, so the matrices $V_+$ and $\G$
are of size $\Np\times \Np$ by construction~(\ref{eq:Vdef}). Note
there is one-to-one correspondence between the zeros of
$\det(F^q(\a))$ in $\Cpos$ and the eigenvalues of $-\Lmb(q)$
when $q>0$.

Assume now that $q=0$. We only need to show that $\det(F(\a))$ has
$\Np-\ind{\kappa\geq 0}$ zeros (counting multiplicities) in
$\Cpos$. Pick a sequence of $q_n$ converging to $0$ and consider a
sequence of matrix exponents $F^{q_n}(\a)=F(\a)-q_n\matI$ and
transition rate matrices $\Lambda(q_n)$.
From~(\ref{lambda_matrix}) it follows that
$e^{\Lmb(q_n)}\rightarrow e^{\Lmb}$, hence the eigenvalues of
$\Lmb(q_n)$ converge to the eigenvalues of $\Lmb$ (preserving
multiplicities) as $n\rightarrow\infty$. Moreover, all the
eigenvalues of $\Lambda$ have negative real part except a simple
one at 0 if $\kappa\geq 0$. The above mentioned one-to-one
correspondence and the convergence statement of ~\cite[Thm.\
10]{zeros} conclude the proof.
\end{proof}

The above proof strengthens~\cite[Thm.\ 2]{zeros}; we remove the
assumption that $\kappa$ is non-zero and finite.
\begin{corollary}\label{thm_zeros}
It holds that $\det(F(\a))$ has $\Np-\ind{\kappa\geq 0}$ zeros
(counting multiplicities) in $\Cpos$.
\end{corollary}

\section{The Technique}
As an important remark, Thm~\ref{thm_main} can be used as the basis for a technique, which allows one to obtain various useful identities. This technique is used in Sections~\ref{sec:sp_pos_MAP}, Section~\ref{sec:MMBM}, Section~\ref{sec:scale_functions} and can be recovered with some modifications in Section~\ref{sec:integral_equation}.
In short, it consists of the following steps:
\begin{itemize}
  \item use a martingale argument to arrive at an initial
  equation involving the unknown quantities and $F(\a)$,
  \item use the properties of Jordan chains such as Prop.~\ref{jordan_lemma} and Lemma~\ref{jordan_lemma_2} to rewrite the initial equation in terms of $(V,\Gamma)$,
  \item use the special structure of $(V,\Gamma)$, such as
  invertibility of $V$, to simplify the equation,
  \item eliminate the Jordan pair by introducing $\Lambda$ and $\Pi$
  using Thm.~\ref{thm_main} to recover the probabilistic interpretation of the involved matrices and claim uniqueness of the solution.
\end{itemize}
It is noted that our technique can be seen as the extension of the
ideas known as `martingale calculations for MAPs'~\cite[Ch. XI,
4a]{asmussen:apq} to its final and general form. It is important
that no assumptions about the number and simplicity of the
eigenvalues are needed. Moreover, using this technique based on
generalized Jordan chains, one can eliminate Jordan pairs from the
solution and claim uniqueness.

The procedure is illustrated in the next sections.

\section{One-sided Reflection Problems}
The first illustration of our technique concerns the stationary process of the
reflection of spectrally one-sided MAPs.
For a given MAP $(X(t),J(t))$ the {\it reflected process}
$(W(t),J(t))$ is defined through
\begin{equation}W(t)=X(t)-\inf_{0\leq s\leq t}X(s)\wedge 0;\end{equation}
we say that $W(t)$ is the (single-sided) reflection of $X(t)$ at 0.
It is well known that this process has a unique stationary distribution if the
asymptotic drift $\kappa$ is negative, which we
assume in the sequel. Let a pair of random variables $(W,J)$ have
the stationary distribution of $(W(t),J(t))$, and denote
the all-time maximum attained by $X(t)$ through
$\overline{X}=\sup_{t\geq 0}X(t)$. It is an immediate consequence
of \cite[Prop.\ XI.2.11]{asmussen:apq} that
\begin{equation}\label{maxima}
\left(W\mid J=i\right)\hbox{ and }\left(\overline{\hat X}\mid\hat
J(0)=i\right)\hbox{ have the same distribution},
\end{equation}
where $(\hat X(t),\hat J(t))$ is the time-reversed process
characterized by the matrix exponent $\hat
F(\a)=\Delta_\bpi^{-1}F(\a)\t\Delta_\bpi$, where
for any given vector $\bx$ we define $\Delta_{\bx}=\diag[x_1, \ldots, x_N]$.
In the following we identify the distribution of the random variables appearing
in~(\ref{maxima}) for two important classes: spectrally negative
MAPs and spectrally positive MAPs.
\subsection{Spectrally negative MAP}
Let $(X(t),J(t))$ be a spectrally negative MAP with negative
asymptotic drift: $\kappa<0$. It is crucial to observe that
$\left(\overline{X}\mid J(0)=i\right)$ is the life-time of
$J(\tau_x)$, thus it has a phase-type distribution~\cite[Section
III.4]{asmussen:apq} with transition rate matrix $\Lmb$, exit
vector $-\Lmb \1_+$ and initial distribution given by $\be_i\t \, \Pi$. It
is noted that if $i\in\Ed$, then $X(t)$ never hits the interval
$(0,\infty)$ with probability $\p_i(\tau_0=\infty)$, hence
$\left(\overline{X}\mid J(0)=i\right)$ has a mass at zero.

The time-reversed process is again a spectrally negative MAP with
negative asymptotic drift and $\Ed$ being the set of indices of
associated downward subordinators. Thus, with self-evident notation,
the vector of densities of $\left(W\mid J\right)$ at $x>0$ is given by
\begin{equation}{\bs p}(x)=\hat\Pi e^{\hat{\Lmb} x}(-\hat{\Lmb} \1_+).\end{equation}
We conclude that we can express the distribution of the stationary
workload in terms of quantities that follow uniquely from Thm.\
\ref{thm_main}.

\subsection{Spectrally positive MAP}\label{sec:sp_pos_MAP}
Let $(Y(t),J(t))$ be a spectrally positive MAP with
negative asymptotic drift. Define $X(t)=-Y(t)$ and note that
$(X(t),J(t))$ is a spectrally negative MAP with positive
asymptotic drift: $\kappa>0$. Let $F(\a)$ and $\Ep\cup \Ed$ be the
matrix exponent and the partition of the state space of the latter process.
The Laplace-Stieltjes transform of $(W,J)$ is identified
in~\cite{asm_kella_multidimensional} up to a vector $\bl$ of
unknown constants:
\begin{equation}\label{eq:spec_positive}\e [e^{-\a W}\be_J\t]=\a
\bl\t F(\a)^{-1},\end{equation} where $\bl_\downarrow=\0$,
$\bl\t{\1}=\kappa$.

We determine the vector $\blp$ as follows. Let
$\bv^0,\ldots,\bv^{r-1}$ be a Jordan chain of $F(\a)$ associated
with an eigenvalue $\lmb \in\Cpos$. Right multiply both sides of
Eqn.~(\ref{eq:spec_positive}) by $F(\a)$ and use
Prop.~\ref{jordan_lemma} to see that $\lmb \bl\t\bv^j +
\bl\t\bv^{j-1} = 0$ for all $j = 0,\ldots,r-1$; where $\bv^{-1} =
\0$. It trivially follows that $\bl\t V = \kappa\be\t_1$ or
equivalently
\begin{equation}\label{eq:sp_pos_MAP}(\blp)\t=\kappa\be\t_1 (V_+)^{-1}.\end{equation}
It is easy to verify now using Thm.~\ref{thm_main} that
\begin{equation}\label{eq:idnty1}\blp=\kappa\bpi_{\Lmb},\end{equation}
where $\bpi_{\Lmb}$ is the stationary distribution of $\Lmb$, cf.\ \cite[Lemma 2.2]{extremes_dieker}.
Observe that we again succeeded in expressing the  distribution of the stationary workload in terms of quantities that can be determined by applying Thm.\ \ref{thm_main}.

%=========================================================================================================

\section{Two-sided Reflection of MMBM}\label{sec:MMBM}
In this section we consider a Markov-modulated Brownian motion %(MMBM)
$(X(t),J(t))$ of dimension $N$. In this case the matrix exponent (\ref{generator}) has the
special form of a matrix polynomial of second order, that is,
\begin{equation}\label{MMBMgenerator}
 F(\a)=\frac{1}{2}\Delta^2_\bsg \a^2+\Delta_\ba \a+Q,
\end{equation}
where $\a_i\in\R$, $\sigma_i\geq0$ for $i=1,\ldots,N$.

We are interested in a two-sided reflection~\cite{asm_kella_multidimensional} of
$(X(t),J(t))$ with 0 and $b>0$ being lower and upper barriers respectively. This reflected process can be interpreted as a workload process of a queue fed by MMBM, where $b$ is the capacity of the buffer. Let the pair of random variables $(W,J)$ have
the stationary distribution of this process. It is shown
in~\cite{asm_kella_multidimensional} that
\begin{equation}\label{MMBM.mom.gen.func}
    \e\left[e^{\a W}\be_J\t\right]\cdot F(\a)=\a(e^{\a b }\bu\t-\bl\t),
\end{equation}
where $\bu$ and $\bl$ are column vectors of unknown constants. In the
following we show how to compute these constants.
This result completes the investigation on MMBM contained
in~\cite{asm_kella_multidimensional} and extends the previous
works \cite{karandikar:kulkarni:1995,kella_stadje,rogers}.

In order to uniquely characterize the vectors $\bu$ and $\bl$ we exploit the
special structure of MMBM. As MMBM is a continuous
process almost surely,  both $(X(t),J(t))$ and
$(-X(t),J(t))$ are spectrally negative MAPs. Regarding the downward subordinators of $(X(t),J(t))$, let
the sets $E_+$ and $E_\downarrow$, and the cardinalities $N_+$ and $N_\downarrow$ be defined as before.
The downward
subordinators of the process $(-X(t),J(t))$  correspond to
the upward subordinators of the process $(X(t),J(t))$.
Denote the subset of states of $E$ associated to such processes by
$\Eu$. Similarly we denote by $\Em$ the set of the states where
the process $(-X(t),J(t))$ can reach positive records. We
note that it is possible to have that $\Eu\cap\Ed \neq \emptyset$,
as the intersection is given by the states where the process
stays constant. Finally the cardinalities of $\Em$ and $\Eu$ are
denoted through $\Nm$ and $\Nu$ respectively.

Let $\lmb_0,\ldots,\lmb_k$ be the zeros of $\det(F(\a))$ in $\C$ with $\lmb_0=0$.
 Let also $(V_i,\Gamma_i)$ be a Jordan pair of $F(\a)$ corresponding to $\lmb_i$. Define
 \begin{equation}
 V=(V_0,\ldots,V_k),\hspace{0.2in}\Gamma=\diag(\Gamma_0,\ldots,\Gamma_k).
 \end{equation}
\begin{theorem}\label{thm_MMBM_matpol}
The unknown vectors $\bu$ and $\bl$ in Eqn.\ $(\ref{MMBM.mom.gen.func})$ can be uniquely identified in the following way:
$\bud=\0$ and $\blu=\0$ while the vectors $\bup$ and $\blm$ are the solutions of the system of equations
\begin{equation}\label{MMBM_matpol}
   (\bup\t,\blm\t)
    \left(
        \begin{array}{c}
            V_+ \, e^{b\,\G} \\
            -V_-
        \end{array}
    \right)
    = (\bk\t,0,\ldots,0),
\end{equation}
where \begin{equation}\label{dfn:k}\bk\t =   \bpi\t
(\Delta_\ba,\frac{1}{2} \Delta^2_\bsg)\left( \begin{array}{c}
                       V_0\\
                       V_0 \, \G_0\\
                     \end{array}\right)
\end{equation}
is a vector with dimension equal to the multiplicity of the $0$ root of $\det(F(\alpha))$.
\end{theorem}
%===================================================

Before we give a proof of Thm.~\ref{thm_MMBM_matpol} we provide
some comments on the structure of the pair $(V,\Gamma)$. The
following simple lemma identifies the number of zeros of
$\det(F(\a))$ in different parts of the complex plane, see
also~\cite{karandikar:kulkarni:1995} for the case when $\kappa\neq 0$.
\begin{lemma}\label{lem_zeros}
$\det(F(\a))$ has $\Np-\ind{\kappa\geq 0}$ zeros in $\Cpos$,
$\Nm-\ind{\kappa\leq 0}$ zeros in $\Cneg$ and a zero at 0 of
multiplicity $1+\ind{\kappa=0}$.
\end{lemma}
\begin{proof}
It is easy to see that $\det(F(\a))$ is a polynomial of degree
$\Np+\Nm$, hence the total number of zeros of $\det(F(\a))$ in
$\C$ counting their multiplicities is $\Np+\Nm$. On the other hand
Corollary~\ref{thm_zeros} states that $\det(F(\a))$ has
$\Np-\ind{\kappa\geq 0}$ zeros in $\Cpos$ and $\Nm-\ind{\kappa\leq
0}$ zeros in $\Cneg$, because $F(-\a)$ is the matrix exponent of a spectrally
negative process $(-X(t),J(t))$ having asymptotic drift $-\kappa$.
The result follows from~\cite{zeros}, where it is shown that
$\det(F(\a))$ does not have zeros on the imaginary axis which are
distinct from 0.
\end{proof}

Next we note that the null-space of $F(0)=Q$ is spanned by~$\1$,
because $Q$ is an irreducible transition rate matrix. Moreover,
scaling matrix $V$ amounts to scaling both sides of
Eqn.~(\ref{MMBM_matpol}) by the same constant, hence we can
assume that the first vector in $V_0$, and thus also in $V$,
is~$\1$. If $\kappa=0$ then according to Lemma~\ref{lem_zeros} the
zero $\lmb_0$ has multiplicity~2, in which case we have that $\Gamma_0$ is a $2\times 2$ matrix, and $V_0$ an $N\times 2$ matrix, given by
\begin{equation}
\Gamma_0=\left(\begin{array}{cc} 0 & 1 \\ 0 & 0
\end{array}\right),\:\:\:\: V_0=(\1,\bh),\end{equation} where the vector $\bh$ solves
\begin{equation}\label{vec_h}F(0)\bh+F'(0)\1=Q\bh+\Delta_\ba\1=\0\end{equation} due to~(\ref{dfn:jordan}). Since $\bpi\t\Delta_\ba\1$ equals the asymptotic drift $\kappa$,
Eqn.~(\ref{dfn:k}) reduces to
\begin{equation}\label{eq:k}
\bk\t=
\left\{\begin{array}{ll}
\kappa, &\hbox{if }\kappa\neq0\\
(0,\bpi\t(\frac{1}{2}\Delta^2_{\bs\sigma}\1+\Delta_\ba\bh )),  &\hbox{if }\kappa=0.
\end{array}\right.
\end{equation}

We now prove the following technical lemma that further specifies $\bk$ in the case $\kappa=0$ .
\begin{lemma}\label{lem:vec_h}
If $\kappa=0$, then $V_0=(\1,\bh)$ and
\begin{equation}\label{eq:vec_h}b\bu\t\1+(\bu\t-\bl\t)\bh=\bpi\t(\frac{1}{2}\Delta^2_{\bs\sigma}\1+\Delta_\ba\bh)\not=0.\end{equation}
\end{lemma}
\begin{proof}
Differentiating Eqn.~(\ref{MMBM.mom.gen.func}) at $0$ and
right multiplying by $\bh$, we obtain the identity
\begin{equation}
 (\bu\t-\bl\t)\bh=\e[W \be_J\t] Q \bh + \e[\be_J\t] \Delta_\ba \bh = - \e[W \be_J\t] \Delta_\ba \1 + \e[\be_J\t] \Delta_\ba \bh,
 \end{equation}
where the second equality follows from~(\ref{vec_h}). Differentiating Eqn.~(\ref{MMBM.mom.gen.func}) twice at $0$ and
multiplying by $\1$, we find
\begin{equation}
b \,\bu\t \1 =\e[W \be_J\t] \Delta_\ba \1 + \e[\be_J\t] \frac{1}{2} \Delta^2_\bsg \1,
\end{equation}
which summed with the previous equation gives~(\ref{eq:vec_h}).
We conclude the proof by showing that the resulting expression cannot equal $0$.

It is known, see \cite{gohberg:lancaster:rodman:1982}, % section 1.6 page 32 and Thm 7.1 section 7.1
that the maximum length of a Jordan chain cannot exceed the
algebraic multiplicity of the associated eigenvalue. We therefore have that for
any vector $\bv$ it holds that
\begin{equation}
    \frac{1}{2}\Delta^2_\bsg \, \1 + \Delta_\ba \, \bh + Q \bv \neq \0,
\end{equation}
because otherwise $(\1,\bh,\bv)$ would be a Jordan chain
associated with $\lmb_0$, which has multiplicity 2 (Lemma \ref{lem_zeros}). This implies
that $\frac{1}{2}\Delta^2_\bsg \, \1 + \Delta_\ba \, \bh$ is not
in the column space of $Q$, which is known to be of dimension
$N-1$, because $Q$ is irreducible. Moreover, $\bpi\t Q=\0\t$,
thus $\bpi\t(\frac{1}{2}\Delta^2_\bsg \, \1 + \Delta_\ba \,
\bh)\neq 0$.
\end{proof}

We now construct pairs $(V^+,\Gp)$ and $(V^-,\Gm)$ in the same
way as we constructed $(V,\Gamma)$, but we use only those $(V_i,\Gamma_i)$ for
which $\lmb_i\in\Cpos$ and $\lmb_i\in\Cneg$, respectively.
Moreover, an additional pair $(\1,0)$ is used (as the first pair) in the
construction of $(V^+,\Gp)$ and $(V^-,\Gm)$ if
$\kappa\geq 0$ and $\kappa\leq 0$, respectively. Note that in
view of Lemma~\ref{lem_zeros} the matrix $V^\pm$ has exactly $N_\pm$
columns.

In the following we use $\Lmb^\pm$ and $\Pi^\pm$ to denote the
matrices associated to the first passage process of $(\pm
X(t),J(t))$, and defined according to
Eqns.~(\ref{lambda_matrix}) and~(\ref{pi_matrix}). Next we
present a consequence of Thm.~\ref{thm_main}.
\begin{lemma}\label{main_reformulated}
The following holds
\begin{equation}
\begin{array}{ll}
\Lmbp  = -\Vpp\Gp(\Vpp)^{-1}                &\quad\Pi^+ = V^+(\Vpp)^{-1}\\
\Lmbm = \phantom{-}\Vmm\Gm(\Vmm)^{-1}   &\quad\Pi^-  = V^-(\Vmm)^{-1}.
\end{array}
\end{equation}
\end{lemma}
\begin{proof}
The first line is immediate from Thm.~\ref{thm_main} and the second
line follows by noting that if $\bv^0,\ldots,\bv^{r-1}$ is a
Jordan chain of $F(\a)$ then
$\bv^0,-\bv^1,\ldots,(-1)^{r-1}\bv^{r-1}$ is a Jordan chain of
$F(-\a)$. Lemma~\ref{eig_lemma} applied to the process
$(-X(t),J(t))$ entails that $(-\lmb \matI+\Lmb^-)\bv^j_--\bv^{j-1}_-=\0$,
where $\lmb$ is a zero of $\det(F(\a))$ in $\Cneg$. Hence
$\Lmbm \Vmm=\Vmm\Gm$.
\end{proof}

%=======================
We now proceed with the proof of Thm.~\ref{thm_MMBM_matpol}.
\begin{proof}\textit{of Theorem~\ref{thm_MMBM_matpol}\quad}
The proofs of the facts that $\bud=\0$ and $\blu=\0$ and, moreover,
$(\bu\t-\bl\t)\1=\bpi\t\Delta_\ba\1=\kappa$ were already given
in~\cite{asm_kella_multidimensional}. The rest of the proof is
split into two steps. First we show that $(\bu_+\t,\bl_-\t)$
solves~(\ref{MMBM_matpol}), and then we show that the solution is
unique.

\noindent \underline{Step 1}: Lemma~\ref{jordan_lemma_2} and
Prop.~\ref{jordan_lemma} applied to
Eqn.~(\ref{MMBM.mom.gen.func}) imply
\begin{equation}\label{eq:proof_main}\bu\t Ve^{b\Gamma}\Gamma-\bl\t
V\Gamma=\0\t.\end{equation} Let $\hat \Gamma$ be the matrix
$\Gamma$ with Jordan block $\Gamma_0$ replaced with $\matI$.
Suppose first that $\kappa\neq 0$. Then $(V_0,\Gamma_0)=(\1,0)$ and
so~(\ref{eq:proof_main}) can be rewritten as
\begin{equation}
\bu\t Ve^{b\Gamma}\hat\Gamma-\bl\t
V\hat\Gamma=([\bu\t-\bl\t]\1,0,\ldots,0).
\end{equation}
Multiply by $\hat\Gamma^{-1}$ from the right to obtain~(\ref{MMBM_matpol}).
Supposing that $\kappa=0$, we have that adding
\begin{equation}
\hat\bk\t=\bu\t V_0e^{b\Gamma_0}\left(\begin{array}{cc}
  1 & -1 \\
  0 & 1
\end{array}\right)-\bl\t V_0\left(\begin{array}{cc}
  1 & -1 \\
  0 & 1
\end{array}\right)
\end{equation}
to the first two elements of the vectors appearing on the both sides of~(\ref{eq:proof_main}), leads to
\begin{equation}
\bu\t Ve^{b\Gamma}\hat\Gamma-\bl\t V\hat\Gamma=(\hat
\bk\t,0,\ldots,0).
\end{equation}
To complete Step 1 it is now enough to show that
$\hat \bk=\bk$. A simple computation reveals that
$\hat\bk\t=\bu\t(\1,(b-1)\1+\bh)-\bl\t(\1,-\1+\bh)=(0,b\bu\t\1+(\bu\t-\bl\t)\bh)$,
where we used that $(\bu\t-\bl\t)\1=\kappa=0$. Use
Lemma~\ref{lem:vec_h} and~(\ref{eq:k}) to see that $\hat \bk=\bk$.

\noindent \underline{Step 2} (uniqueness): Without loss of generality we
assume that \mbox{$\kappa\geq 0$}. It is easy to see that
$(\bu_+\t,\bl_-\t)$ solves~(\ref{MMBM_matpol}) if and only if
\begin{equation}
(\bu_+\t,-\bl_-\t)\left(\begin{array}{cc}
  V_+^+e^{b\Gp} & V^-_+e^{b\Gm} \\
  V_-^+ & \Vmm
\end{array}\right)=\kappa\be\t_1
\end{equation}
and, in addition when $\kappa=0$ Eqn.~(\ref{eq:vec_h}) holds true
because of the construction of the matrices $V^\pm$ and $\Gamma^\pm$.
Note also that we can right multiply both sides of the above display by the same matrix to obtain
\begin{equation}
(\bu_+\t,-\bl_-\t)\left(\begin{array}{cc}
  V_+^+ & V^-_+e^{b\Gm} \\
  V_-^+e^{-b\Gp} & \Vmm
\end{array}\right)=\kappa\be\t_1\left(\begin{array}{cc}
  e^{-b\Gp} & \matO \\
  \matO & \matI
\end{array}\right).
\end{equation}
Lemma~\ref{main_reformulated} shows that
\begin{equation}
\begin{array}{ll}
V^+_- = \Pi_-^+\Vpp, &\quad\Vpp e^{-b\Gp}=e^{b\Lmbp}\Vpp,\\
V^-_+=\Pi_+^-\Vmm, &\quad\Vmm e^{+b\Gm}=e^{b\Lmbm}\Vmm,
\end{array}
\end{equation}
so we obtain
\begin{equation}\label{eq:Gamma.to.Lambda}
\left(\begin{array}{cc}
  V_+^+ & V^-_+e^{b\Gm} \\
  V_-^+e^{-b\Gp} & \Vmm
\end{array}\right)=\left(\begin{array}{cc}
  \matI & \Pi^-_+e^{b\Lmb^-} \\
  \Pi_-^+e^{b\Lmb^+} & \matI
\end{array}\right)\left(\begin{array}{cc}
  \Vpp & \matO \\
  \matO & \Vmm
\end{array}\right).
\end{equation}
Theorem~\ref{thm_main} states that $\Vpp$ and $\Vmm$ are
invertible matrices. Moreover, $\Pi_-^+e^{b\Lmb^+}$ and
$\Pi^-_+e^{b\Lmb^-}$ are irreducible transition probability
matrices, so the first matrix on the right hand side of
(\ref{eq:Gamma.to.Lambda}), call it $M$, is an irreducible
non-negative matrix, which is non-strictly diagonally dominant. If
$\kappa > 0$, then $\Pi^-_+e^{b\Lmb^-}$ is sub-stochastic, which
implies that $M$ is irreducibly diagonally dominant and hence
invertible~\cite{horn:1990}. If $\kappa=0$, then $M$ has a simple
eigenvalue at 0 by Perron-Frobenius, so
\begin{equation}
(\bu_+\t,-\bl_-\t)\left(\begin{array}{cc}
  \matI & \Pi^-_+e^{b\Lmb^-} \\
  \Pi_-^+e^{b\Lmb^+} & \matI
\end{array}\right)=\0\t
\end{equation}
determines the vector $(\bu_+\t,-\bl_-\t)$ up to a scalar, which is then identified using~(\ref{eq:vec_h}):
\begin{equation}\label{eq:2sided_0drift}
(\bu_+\t,-\bl_-\t)
\left(\begin{array}{c}
  b\1_++\bh_+ \\
  \bh_-
\end{array}\right)
=\bpi\t(\Delta_\ba\bh+\frac{1}{2}\Delta^2_{\bs\sigma}\1),
\end{equation}
which is non-zero by Lemma~\ref{lem:vec_h}.
\end{proof}

Finally, we state a corollary, which identifies (in the case of non-zero asymptotic drift) vectors $\bu_+$ and $\bl_-$ in terms of matrices $\Lmb^\pm$ and $\Pi^\pm$. We believe that there should exist a direct probabilistic argument leading to this identity. Moreover, it is interesting to investigate if such a result holds in the case of countably infinite state space $E$.
\begin{corollary}
It holds that
\begin{equation}\label{eq:idnty2}
(\bu_+\t,-\bl_-\t)\left(\begin{array}{cc}
  \matI & \Pi^-_+e^{b\Lmb^-} \\
  \Pi_-^+e^{b\Lmb^+} & \matI
\end{array}\right)=
\left\{\begin{array}{ll}
\kappa(\bpi\t_{\Lmb^+},\0\t_-),&\hbox{ if }\kappa> 0 \\
\0\t,&\hbox{ if }\kappa=0 \\
\kappa(\0\t_+,\bpi\t_{\Lmb^-}),&\hbox{ if }\kappa< 0
\end{array}\right.,
\end{equation}
where $\bpi_{\Lmb^\pm}$ is the unique stationary distribution of
$\Lmb^\pm$, which is well-defined if $\kappa\not=0$.
\end{corollary}
\begin{proof}
Assume that $\kappa> 0$. From the above proof we know that
\begin{equation}
(\bu_+\t,-\bl_-\t)\left(\begin{array}{cc}
  \matI & \Pi^-_+e^{b\Lmb^-} \\
  \Pi_-^+e^{b\Lmb^+} & \matI
\end{array}\right) \left(\begin{array}{cc}
  \Vpp & \matO \\
  \matO & \Vmm
\end{array}\right)=\kappa(\be\t_1,\0\t_-).
\end{equation}
Hence it is enough to check that
$\bpi_{\Lmb^+}=\be\t_1(\Vpp)^{-1}$, which is immediate in view of
Thm.~\ref{thm_main}. The case of $\kappa<0$ is symmetric, and the
case of $\kappa=0$ is trivial.
\end{proof}
%=================================================================
\section{The Scale Functions Associated to MMBM}\label{sec:scale_functions}
The third application of our technique focuses on so-called scale functions associated to MMBM.
Fix non-negative numbers $a$ and $b$, which are not simultaneously 0, and let
 $\tau_x^-=\inf\{t\geq 0:X(t)<-x\}$ be the first passage over $x$ of the process $-X(t)$. We consider
the following matrices
$$
C(a,b)=\e[e^{-q\tau_a};\tau_{a}<\tau_b^-,J(\tau_{a})] {\rm \ and
}D(a,b)=\e[e^{-q\tau_b^-};\tau_{b}^-<\tau_a,J(\tau_{b}^-)]$$
having dimensions $N\times N^+$ and $N\times N^-$ respectively. These
matrices are in fact generalizations of the scale functions in the
terminology of the theory of \levy processes.

Using the strong Markov property we write
\begin{eqnarray*}
C(a,b)&=&\e [e^{-q\tau_a};J(\tau_a)]-\e [e^{-q\tau_a};\tau_b^-<\tau_a,J(\tau_a)]\\
&=&\e[e^{-q\tau_a};J(\tau_a)]-D(a,b)\e[e^{-q\tau_{a+b}};J(\tau_{a+b})],
\end{eqnarray*}
where under the first expectation it is implicit that $\tau_a<\infty$.
 Thus we get
\begin{eqnarray}\label{eq:C_D}
C(a,b)&=&\Pi^+e^{a\Lambda^+}-D(a,b)\Pi^+_-e^{(a+b)\Lambda^+}\\
\nonumber D(a,b)&=&\Pi^-e^{b\Lambda^-}-C(a,b)\Pi^-_+e^{(a+b)\Lambda^-},
\end{eqnarray}
where $\Pi^\pm$ and $\Lambda^\pm$ should be read as $\Pi^\pm(q)$ and $\Lambda^\pm(q)$. The dependence on $q$ is dropped for notational simplicity.
Similarly as in the proof of Thm.~\ref{thm_MMBM_matpol} one readily sees that these equations determine the scale functions uniquely
unless $q=0$ and $\kappa=0$, in which case one additional equation is required. In the following we show how our technique can
be applied here to recover this result, as well as the missing
equality in the case of $q=0,\kappa=0$.

Pick an arbitrary eigenvalue $\lmb$ of $F(\a)-q\matI$ and a
corresponding Jordan chain $\bv_0,\ldots,\bv_{r-1}$. Following the
same steps as in the proof of Lemma~\ref{eig_lemma}, but using
$\tau=\tau_a\wedge\tau^-_b$ instead of $\tau_x$ we obtain
\[\sum_{i=0}^j\frac{1}{i!}\e\left[X(\tau)^ie^{\lambda X(\tau)};\tau<e_q,J(\tau)\right]\bv^{j-i}-\bv^j=\0,\]
which further reduces to
\begin{equation}\label{eq:jordan_1}\nonumber\sum_{i=0}^j\frac{1}{i!}\left[a^ie^{\lambda
a}C(a,b)\bv^{j-i}_++(-b)^ie^{-\lambda
b}D(a,b)\bv_-^{j-i}\right]-\bv^j=\0.\end{equation} A simple
observation presented in Lemma~\ref{jordan_lemma_2} shows that
\begin{equation}\label{eq:jordan}C(a,b)V_+ e^{a\Gamma}+D(a,b)V_-
e^{-b \Gamma}=V,\end{equation} where $(\Gamma,V)$ is a Jordan pair
corresponding to $F(\a)-q\matI$.

Restricting ourselves to the Jordan pair $(V^+,\Gamma^+)$ and
multiplying both sides of~(\ref{eq:jordan}) from the right by
$e^{-a\Gamma^+}{V_+^+}^{-1}$ we obtain
\[C(a,b)=V^+e^{-a\Gamma^+}{V_+^+}^{-1}-D(a,b)V^+_- e^{-(a+b)
\Gamma^+}{V_+^+}^{-1}.\] Finally, application of
Lemma~\ref{main_reformulated} results in the first equation of
(\ref{eq:C_D}), and the second is obtained restricting ourselves
to $(V^-,\Gamma^-)$. In the case $q=0,\kappa=0$ the above
procedure results in the lost of the equation associated to the
null Jordan chain $(\1,\bh)$. Simple computations show that this
equation is
\begin{equation}\label{eq:scale_lost}C(a,b)(a\1_++\bh_+)+D(a,b)(-b\1_-+\bh_-)=\bh,\end{equation} which can
be shown to be linearly independent of the rest using the idea
from the proof of Thm.~\ref{thm_MMBM_matpol}. In conclusion, note
that equation (\ref{eq:jordan}) has a number of advantages over
equation (\ref{eq:C_D}). It is simpler, easier to use in
computations, and, moreover, it has always a unique solution.
%===============================================
\section{Matrix Integral Equation}\label{sec:integral_equation}
In this section we demonstrate how our technique can be used to show in a simple way that $\Lambda(q)$ is a unique
solution of a certain matrix integral equation. This equation appears in~\cite{asmussen:mmbm, extremes_dieker, miyazawa, pacheco_prabhu, pistorius, rogers} and is commonly considered as the main tool in numerically computing $\Lambda(q)$.
Throughout this section it is assumed that $N_+=N$. We use the
following notation
\begin{eqnarray}\label{eq:mat_eq}\nonumber
F^q(M)&=&\Delta_\ba M+\frac{1}{2}\Delta^2_\bsg M^2+\int_{-\infty}^0\Delta_{\bs \nu} (\D x)\left( e^{Mx}-\matI-Mx\ind{x>-1}\right)\\
&\phantom{=}&+\int_{-\infty}^0Q\circ G(\D x)e^{Mx}-q\matI
\end{eqnarray}
where $(a_i,\sigma_i,\nu_i(\D x))$ are the \levy triplets
corresponding to the \levy processes $X_i(\cdot)$ and $G_{ij}(\D
x)$ is the distribution of $U_{ij}$ and $M$ is any given square matrix of size $N$.

Define $\M$ to be a set of all $N\times N$ matrices $Q$, such that
$Q$ is a transition rate matrix of an irreducible Markov chain.
Let also $\{\M_0,\M_1\}$ be a partition of $\M$ into the sets of
defective and non-defective matrices respectively.
\begin{theorem}\label{thm:mat_equation}
$\Lmb(q)$ is the unique solution of $F^q(-M)=\matO$, where
$M\in\M_i$ and $i=\ind{q=0,\kappa\geq 0}$.
\end{theorem}
\begin{proof}
In the proof we drop the superscript $q$ to simplify notation. Let
$-M=V\G V^{-1}$ be a Jordan decomposition of the matrix $-M$. Let
also $\bv_0,\ldots,\bv_{r-1}$ be the columns of $V$ corresponding
to some Jordan block of size $r$ and eigenvalue $\lmb$. Note that
$\lmb \in\Cpos$ or $\lmb=0$ in which case it must be simple, because
$M\in\M$.
 Right
multiply~(\ref{eq:mat_eq}) by $V$, note that
$g(-M)=Vg(\G)V^{-1}$ for an entire function $g:\C\rightarrow\C$,
and finally use Lemma~\ref{jordan_lemma_2} to see that the column
of $F(-M)V$ corresponding to $\bv_j$ equals
\begin{equation}\label{eq:matrix_eq_Lambda}\sum_{i=0}^j \frac{1}{i!}F^{(i)}(\lmb)\bv_{j-i},\end{equation} where
we also used the fact that differentiation of $F(\a)$ at
$\lmb,\Re(\lmb)> 0$ can be done under the integral signs and no
differentiation is needed for a simple eigenvalue $\lmb=0$ if such
exists.

If $M=\Lmb$, then the matrices $V$ and $\G$ can be chosen as
in~(\ref{eq:Vdef}). Hence~(\ref{eq:matrix_eq_Lambda}) becomes
$\0$, because $\bv_0,\ldots,\bv_{r-1}$ is a Jordan chain of
$F(\a)$, see~(\ref{dfn:jordan}). But $V$ is an invertible matrix, so that
$F(-\Lmb)=\matO$.

Suppose now that $F(-M)=\matO$ and $M\in\M_i$. Then the vectors
$\bv_0,\ldots,\bv_{r-1}$ form a Jordan chain of $F(\a)$
corresponding to an eigenvalue $\lmb \in\Cpos$ or $\lmb=0$. Finally,
use Lemma~\ref{eig_lemma} to see that $\Lmb V=-V\G$, and hence we have $M=\Lmb$.
\end{proof}
%================================================================
\iffalse
\appendix
\section{\ }
\begin{lemma}\label{lem:martingale}
For any $i\in E$ and $\a\in\Cpos$ the right-continuous process
\begin{equation}
M_\a(t)=\left[\int_0^{t\wedge e_q} e^{\a X(s)}\be\t_{J(s)}\D s\right]\cdot F^q(\a)+\be\t_i-e^{\a X(t)}\mbox{{\large 1}}_{\{t<e_q\}}\be\t_{J(t)}
\end{equation}
is a row vector valued zero mean martingale under the probability
measure $\p_i$.
\end{lemma}
\begin{proof}
It trivially follows from~\cite[Lemma
2.1]{asm_kella_multidimensional} that
\begin{equation}
M^W(\a,t)=e^{\a X(t)}\ind{t<e_q}\be_{J(t)}\t e^{-F^q(\a)t}
\end{equation}
is a row vector valued martingale under $\p_i$. Repeat the steps of
the proof of~\cite[Thm.\ 2.1]{asm_kella_multidimensional} with
$M^W(\a,t)$ defined above and $Y(\cdot)\equiv 0$ to conclude.
As a remark, the fact that the expectation of $M_\a(t)$ under the initial conditions $(X(0),J(0))=(0,i)$ is null follows directly from applying Dynkin's formula (see \cite{karlin:taylor:1981}) to the function $e^{\a X(t)}\mbox{{\large 1}}_{\{t<e_q\}}\be\t_{J(t)}$ of the killed Markov process $(X,J)$.
\end{proof}
\fi

\section*{Acknowledgments}
The first author is partially supported by the Spanish Ministry of Education and Science Grants MTM2007-63140 and SEJ2007-64500 and by the Community of Madrid Grants CCG08-UC3M/ESP-4162 and CCG07-UC3M/ESP-3389. Part of his research was done when he was visiting the Hebrew University of Jerusalem by partial support of  Madrid University Carlos III Grant for Young Researchers' Mobility. The third author is partially supported by grant 964/06 from the Israel Science Foundation and the Vigevani Chair in Statistics.

\end{document}